\documentclass[english,oneside]{amsart}
\usepackage[T1]{fontenc}
\usepackage[latin1]{inputenc}

\usepackage{xcolor}
\usepackage{amssymb}
\makeatletter

 \theoremstyle{plain}
\newtheorem{thm}{Theorem}[section]
  \theoremstyle{plain}
  \newtheorem{lem}[thm]{Lemma}
 
  \theoremstyle{definition}
  \newtheorem*{example*}{Example}
  \newtheorem*{def*}{Definition}
  \theoremstyle{remark}
  \newtheorem*{rem*}{Remark}
  \newtheorem{rem}[thm]{Remark}
  \theoremstyle{plain}
  \newtheorem{pro}[thm]{Proposition}
  \theoremstyle{plain}
  
  \theoremstyle{plain}
  \newtheorem{cor}[thm]{Corollary}
  \theoremstyle{definition}

\theoremstyle{definition}\newtheorem{defi}[thm]{Definition}
\theoremstyle{definition}
\theoremstyle{definition}\newtheorem{probl}[thm]{Problem}

\DeclareMathOperator{\rank}{rank}
\DeclareMathOperator{\Span}{Span}
\DeclareMathOperator{\at}{at}
\DeclareMathOperator{\coat}{coat}
\DeclareMathOperator{\Cone}{cone}

\usepackage{babel}
\title[Bruhat intervals and Coxeter matroids]{Weak generalized lifting property,\\ Bruhat intervals and Coxeter matroids}
\makeatother
\author{Fabrizio Caselli}\author{Michele D'Adderio}\author{Mario Marietti}

\address{Dipartimento di matematica, Universit\`a di Bologna, Piazza di Porta San Donato 5, 40126 Bologna, Italy}
\email{fabrizio.caselli@unibo.it}

\address{Universit\'e Libre de Bruxelles (ULB)\\D\'epartement de Math\'ematique\\ Boulevard du Triomphe, B-1050 Bruxelles\\ Belgium}
\email{mdadderi@ulb.ac.be}

\address{Dipartimento  di Ingegneria Industriale e Scienze Matematiche, Universit\`a Politecnica delle Marche, Via Brecce Bianche, 60131 Ancona,  Italy}
\email{m.marietti@univpm.it}

\begin{document}
\maketitle

\begin{abstract}We provide a weaker version of the generalized lifting property which holds in complete generality for all finite Coxeter groups, and we use it to show that every parabolic Bruhat interval of a finite Coxeter group is a Coxeter matroid. We also describe some combinatorial properties of the associated polytopes.
\end{abstract}
\section{Introduction}
Standard matroids were introduced by Whitney \cite{Whit} in 1935 as a combinatorial abstraction of the concept of linear dependence of vectors, and since then they have found wide applications in several branches of mathematics. This paper is devoted to a natural generalization of the notion of standard matroid, which has been introduced by Gelfand and Serganova \cite{GS} in 1987 and is known as  Coxeter matroid. Coxeter matroids are associated with a finite Coxeter group $W$ and a  parabolic subgroup $P$ of $W$, and one can recover standard matroids by taking the symmetric group $S_n$ as $W$, together with any maximal parabolic subgroup $P$ of $S_n$.

The original definition of Coxeter matroid is of algebraic combinatorial nature and is given in terms of Bruhat order. Nevertheless, a fundamental result in the theory of Coxeter matroids gives an equivalent definition of geometric combinatorial nature in terms of an associated polytope, which  is called a Coxeter matroid polytope: this is the content of a theorem by Gelfand--Serganova \cite{GS}  and Serganova--Vince--Zelevinsky \cite{SVZ} in different levels of generality (see Theorem \ref{defCM}).

The Kostant--Toda lattice is an integrable Hamiltonian system  which has been recently studied in detail by Kodama and Williams in \cite{KW}; in that paper, the authors study the asymptotic behaviour of the flow corresponding to an initial point associated with (a given point in) a Richardson cell $\mathcal R_{u,v}^+ $ of the totally non negative flag variety, where $[u,v]$ is a Bruhat interval in the symmetric group $S_n$.   Considering a natural multivariable generalization of this problem, called the full Kostant--Toda hierarchy, they also prove that the moment polytope associated to the flow with  initial point corresponding  to $\mathcal R_{u,v}^+ $ is exactly the Coxeter matroid polytope associated with the Bruhat interval $[u,v]$. This polytope is called a Bruhat interval polytope and their construction implies in particular that every Bruhat interval in $S_n$ is actually a Coxeter matroid. 

In \cite{TW}, Tsukerman and Williams analyze some combinatorial aspects of a Bruhat interval polytope of the symmetric group, and, in particular, they find a dimension formula, prove that every face of a Bruhat interval polytope is itself a Bruhat interval polytope, and extend some of these results to the class of Weyl groups. The key fact that they use in the  study of the Bruhat interval polytope is what they call the generalized lifting property, which generalizes the (standard) lifting property for the symmetric group. This generalized lifting property asserts that, for every $u<v$ in the symmetric group and every ``minimal'' reflection (i.e. a transposition) $t$ such that $u<ut$ and $v>vt$, we have $u\lhd ut\leq v$ and $u\leq vt\lhd v$, where the symbol $\lhd$ denotes the covering relation in Bruhat order. The main fact about this property is that such minimal reflection always exists. In \cite{CS}, the generalized lifting property is proved to hold true for a Coxeter group $W$  if and only if $W$ is finite and simply laced.

In this paper, we first provide a weak version of the generalized lifting property, which on the contrary holds for all finite Coxeter groups and may have interest in its own right. The proof of this property is based on the theory of reflection orderings on the set of positive roots in every Coxeter system. Then we extend the results of Tsukerman and Williams to the context of all finite Coxeter groups and in particular we show  that all parabolic Bruhat intervals in finite Coxeter groups are actually Coxeter matroids as a consequence of the weak generalized lifting property. Furthermore, we show that, also in this level of generality, faces of Bruhat interval polytopes are themselves Bruhat interval polytopes. The proof is first established in the standard, i.e. non-parabolic, case, and then it is deduced in the general parabolic case using some tools of convex geometry that we develop. We also provide, in the standard case, a characterization of those subintervals which give rise to a face.

All results in this paper are valid at least in the generality of  all finite Coxeter groups. The corresponding proofs are mainly of combinatorial nature and within the framework of Coxeter group theory; we also point out that all proofs  make no use of the classification of these groups. 



\section{Preliminaries and background}

In this section, we collect  some notation, definitions,
 and results that are used in the rest of this work. We assume knowledge of the basic definitions of poset theory (see \cite{StaEC1}) and Coxeter group theory (see \cite{BB} or \cite{Hum}).

We  follow \cite[Chapter 3]{StaEC1} for undefined notation and 
terminology concerning partially ordered sets. In particular,
 given $x$ and $y$ in a partially ordered set  $P$, we say that $y$ {\em covers} $x$ (or $x$ is {\em covered} by $y$) and we write $x \lhd y$ if the interval $ \{ z \in P\colon\, x \leq z \leq y \}$, denoted $[x,y]$, has two elements,  $x$ and $y$. An element $z $ in an interval $[u,v]$ is said to be an {\em atom} (respectively, a {\em coatom}) of $[u,v]$ if $u \lhd z$ (respectively, $z \lhd v$).  The {\em Hasse diagram} of $P$ is any drawing of the graph having $P$ as vertex set and $ \{ \{ x,y \} \in \binom {P}{2} \colon\,  \text{ either $x \lhd y$ or $y \lhd x$} \}$ as edge set, with the convention that, if $x \lhd y$, then the edge $\{x,y\}$ goes upward from $x$ to $y$. When no confusion arises, we make no distinction between the Hasse diagram and its underlying graph.


Given a Coxeter system $(W,S)$, we denote by $e$
the identity of $W$, and we let $T =\{ w s w ^{-1} \colon\,  w \in W, \; s \in S \}$ be the set of {\em reflections} of $W$. The Coxeter group $W$ is partially ordered by {\em Bruhat order} (see, for example,  \cite[\S 2.1]{BB} or \cite[\S 5.9]{Hum}), which is denoted by $\leq$. The Bruhat order is the partial order whose  covering relation $\lhd$ is as follows: if $u,v \in W$, then $u \lhd v$ if and only if $uv^{-1} \in T$ and $\ell (u)=\ell (v)-1$. 

For later reference, we state the following well-known result, which is implied by the fact that the order complex of a Bruhat interval is shellable (see \cite{BW}).
\begin{pro}
\label{collegate}
Let $(W,S)$ be a Coxeter system, and $u,v\in W$ with $u\leq v$. Let $C$ and $C'$ be two complete chains from $u$ to $v$. Then there exist $k\in \mathbb N$ and a sequence $C=C_1, C_2, \ldots, C_k=C'$ of complete chains from $u$ to $v$ such that $C_i$ and $C_{i+1}$ differ by one element, for all $i=1,2, \ldots, k-1$.  
\end{pro}

We denote by $V$ the finite dimensional real vector space that realizes the standard geometric representation  $W\rightarrow GL(V)$ of the Coxeter system $(W,S)$ (see \cite[\S 4.2]{BB}). We denote the set of roots by $\Phi$ and the set of positive roots by $\Phi^+$. We write $t_{\alpha}$ for the reflection corresponding to the positive root $\alpha$ and, conversely, $\alpha_{t}$ for the positive root corresponding to the reflection $t$. For every covering relation $x\lhd y$ in $W$, we denote by $\alpha_{x\lhd y}$ the root corresponding to the reflection $t$ such that $y=tx$. 

A \emph{reflection ordering} is a total ordering $\preceq $ on $\Phi^+$ provided that, for all $\alpha, \beta \in \Phi^+$ and $a,b \in \mathbb R_{>0}$ such that $a\alpha+  b\beta \in \Phi^+$, either $\alpha \preceq a\alpha+  b\beta \preceq \beta$ or $\beta \preceq a\alpha + b\beta \preceq \alpha$. Equivalently, a reflection ordering is a total ordering $\preceq $ on $T$ provided that, for all $t, r \in T$, either $t \preceq trt \preceq trtrt  \preceq  \cdots \preceq rtrtr \preceq rtr \preceq r$ or $ r \preceq rtr \preceq rtrtr  \preceq  \cdots \preceq trtrt \preceq trt \preceq t$ (where all expressions are reduced). 
By the definition of Bruhat order, if $u,v\in W$ with $u\lhd v$, then $v=tu$ for a certain $t\in T$: in this case, we label the edge $\{u,v\}$ of the Hasse diagram of $W$ with $t$. Fix a reflection ordering $\preceq$, and let $x,y \in W$ with $x\leq y$. We say that a complete chain $x \lhd t_1x \lhd t_2t_1x \lhd \cdots \lhd t_n \cdots t_1 x=y$ from $x$ to $y$ is \emph{increasing} if $t_1 \preceq t_2\preceq \cdots \preceq t_n$, i.e., if the labels of the edges of the chain are increasing in the fixed reflection ordering from bottom to top.
We state the following result for later reference (see \cite[Proposition~4.3]{D1993}). 
\begin{pro}
\label{unasola}
Let $(W,S)$ be a Coxeter system, and fix any reflection ordering. If $u,v\in W$ with $u\leq v$, then  there is a unique increasing complete chain from $u$ to $v$. 
\end{pro}

For each $w\in W$, the \emph{$w$-Bruhat order}, denoted $\leq ^w$, is the partial order on $W$ defined as follows: 
$$u \leq ^w v \iff w^{-1}u \leq w^{-1}v.$$
Clearly, $\leq ^e$ coincides with the Bruhat order $\leq$.

For each subset $J\subseteq S$, we denote by  $W_J $ the parabolic subgroup of $W$ generated by $J$,  by $W/W_J$ the set of left cosets of $W_J$, and by  $W^{J}$ the set of minimal left coset representatives:
$$W^{J} =\{ w \in W \colon\,  ws>w \textrm{ for all }s\in J \}.$$

The following classical result (see, for example,  \cite[\S 2.4]{BB} or \cite[\S 1.10]{Hum}) implies that the set $W^J$ is actually a set of left coset representatives of $W_J$. 
\begin{pro}
\label{fattorizzo}
If  $J \subseteq S$, then
\begin{enumerate}
\item[(i)] 
every $w \in W$ has a unique factorization $w=w^{J}  w_{J}$ 
with $w^{J} \in W^{J}$ and $w_{J} \in W_{J}$,
\item[(ii)] for this factorization, $\ell(w)=\ell(w^{J})+\ell(w_{J})$.
\end{enumerate}
\end{pro}

Given two left cosets $A,B\in W/W_J$ and $w\in W$, we let $A\leq^w B$ provided that there exist $u\in A$ and $v\in B$ such that $u \leq^w v$. For short, we write $A\leq B$ to mean $A\leq^e B$.

The following is a reformulation of \cite[Theorems~5.14.5 and 5.14.6]{BGW}.
\begin{thm}
\label{ecco}
Let $(W,S)$ be finite, $J\subseteq S$, and $w\in W$. For each $A\in W/W_J$, there exist elements $\min^w A, \max^w A\in A$ such that $\min^w A\leq^w u\leq^w \max^w A$, for all $u\in A$.
Furthermore, if $A,B\in W/W_J$, then the following are equivalent:
\begin{enumerate}
\item $A\leq^w B$,
\item $\min^w A\leq^w \min^w B$,
\item $\max^w A\leq^w \max^w B$.
\end{enumerate}
\end{thm}

Let $W$ be a finite Coxeter group, which we identify with the finite reflection group given by its standard geometric representation on the vector space $V$.  For later references, we state explicitly the following well-known and easy fact. 
\begin{lem}
\label{anticamera}
Let $(W,S)$ be a finite Coxeter system, and $p\in V$ be such that  
$(p,\alpha_{s})<0$ for all $s\in S$. Let $x,y \in W$. If $x\lhd y$, then there exists $c >0$ such that $$y(p)= x(p) + c \alpha_{x\lhd y}.$$
\end{lem}
Attached with every set of left cosets $\mathcal M\subseteq W/W_J$, there is a polytope $\Delta_{\mathcal M}$, which is defined as follows. Fix a base point $p\in V$ that satisfies 
$(p,\alpha_{s})
  \begin{cases}
 <0 &  \text{for $s\notin J$}\\
 =0 &  \text{for $s\in J$}.
\end{cases}
$
Since $W_J$ fixes $p$, we can define a map $\delta_p: W/W_J \to V$ sending the left coset $vW_J$ to $v(p)$. 
\begin{defi}
\label{politopo}
The polytope $\Delta_{\mathcal M}(p)$ is defined as the convex hull of the points in $\delta_p(\mathcal M)$. \end{defi}
\begin{rem}
\label{vero}
Since the action of $W$ on $\delta_p(W/W_J)$ is transitive,  each point in $\delta_p(\mathcal M)$ is a vertex of  $\Delta_{\mathcal M}(p)$, for all $\mathcal M \subseteq W/W_J$.
\end{rem}
\begin{rem}
\label{vero2}
A theorem of Borovik and Vince \cite[Theorem 5.5]{BV} states that the combinatorial type (i.e. the isomorphism class of the simplicial complex of the faces) of the polytope $\Delta_{\mathcal M}(p)$ does not depend on the chosen point $p$. This is the reason why it is customary to drop the dependence on $p$ and simply write $\Delta_{\mathcal M}$.
\end{rem}
The following result is part of the most important result in the theory of Coxeter matroids and is known as the Gelfand--Serganova Theorem (see \cite[Theorem 5.2]{SVZ} or \cite[\S 6.3]{BGW}).
\begin{thm}
\label{defCM}
Let $(W,S)$ be a finite Coxeter system, $J\subseteq S$, and $\mathcal M \subseteq W/W_J$ with $\mathcal M\neq \emptyset$. The following are equivalent:
\begin{enumerate}
\item for every $w\in W$, there is a unique $A\in \mathcal M$ such that $B\leq^w A$ for all $B\in \mathcal M$ (Maximality Property),
\item every edge of the polytope $\Delta_{\mathcal M}$ is parallel to a root in $\Phi$.
\end{enumerate}
\end{thm}
A subset $\mathcal M \subseteq W/W_J$ is a \emph{Coxeter matroid} provided that the clauses of Theorem~\ref{defCM} are satisfied.

Every finite Coxeter group $W$ has a longest element $w_0$. Since the multiplication by $w_0$ is an order-reversing bijection of $W$, the Maximality Property is also equivalent to the Minimality Property: for every $w\in W$, there is a unique $A\in \mathcal M$ such that $A\leq^w B$ for all $B\in \mathcal M$.

\section{A weak generalized lifting property}

Let $(W,S)$ be a Coxeter system and  $W\rightarrow GL(V)$ be its standard geometric representation. Let $\{\alpha_s:\, s\in S\}$ be the basis of simple roots. If $\alpha \in \Phi^+$ with  $\alpha=\sum_{s\in S}a_s\alpha_s$ ($a_s\in \mathbb R_{\geq0}$), then we let  
\[
\bar \alpha=\frac{1}{\sum_{s\in S} a_s}\alpha.                   
\]


The following result is a generalization of \cite[Proposition~5.2.1]{BB}.
\begin{pro}
\label{reflecorder}
 Let $f_1,\ldots,f_n$ be a basis of $V^*$. We define a total ordering $\preceq $ on $\Phi^+$ in the following way: for $\alpha,\beta\in \Phi^+$, we let $\alpha \preceq \beta$ if 
 \[
  (f_1(\bar \alpha),\ldots,f_n(\bar \alpha))\leq_{lex} (f_1(\bar \beta),\ldots,f_n(\bar \beta)),
 \]
where $\leq_{lex}$ denotes the lexicographical order.
The ordering $\preceq$ is  a reflection ordering.
\end{pro}
\begin{proof}
 Let $\alpha, \beta, \gamma \in \Phi^+$, and $\gamma=a\alpha+b \beta$ with $a,b \in \mathbb R_{>0}$. Clearly  $\bar \gamma=c\bar \alpha+(1-c)\bar \beta$, for some $c\in (0,1)$.
 If $f_i(\bar \alpha)=f_i(\bar \beta)$ then 
 \[
  f_i(\bar \gamma)=cf_i(\bar \alpha)+(1-c)f_i(\bar \beta)=f_i(\bar \alpha)=f_i(\bar \beta),
 \]
 and if $f_i(\bar \alpha)<f_i(\bar \beta)$ then

\[
  f_i(\bar \gamma)=cf_i(\bar \alpha)+(1-c)f_i(\bar \beta) \begin{cases}
                                                             >cf_i(\bar \alpha)+(1-c)f_i(\bar \alpha)=f_i(\bar \alpha)\\<cf_i(\bar \beta)+(1-c)f_i(\bar \beta)=f_i(\bar \beta).
                                                            \end{cases}
\]
The result follows.
\end{proof}

We say that a subset $C$ of $V$ is a \emph{cone} if it is closed under nonnegative linear combinations of its elements. For $A\subset V$, we let $\Cone (A)$ be the set of all finite nonnegative linear combinations of elements in $A$. 
\begin{cor}

\label{C1C2}
 Let $C_1$ and $C_2$ be two cones in $V$. If $C_1\cap C_2=\{0\}$, then there exists a reflection ordering such that every root contained in $C_1$ is smaller than every root contained in $C_2$.
\end{cor}
\begin{proof}
By Minkowski's Hyperplane Separation Theorem (see, for example, \cite[\S 2.5.1]{BV}), there exists a hyperplane that separates $C_1$ and $C_2$. This means that there exists $f\in V^*$ such that  $C_1\setminus\{0\}$ is contained in the open halfspace of equation $f< 0$ and $C_2\setminus\{0\}$ is contained in the open halfspace $f>0$. Now let $f_2,\ldots,f_n\in V^*$ be such that $\{f_1=f,f_2\ldots,f_n\}$ form a basis of $V^*$. With these data,  the order of Proposition~\ref{reflecorder} satisfies the required property.
\end{proof}
In \cite{TW}, Tsukerman and Williams prove the following result, showing that the Coxeter groups of type A satisfy what they call the generalized lifting property (see \cite[Theorem~3.3, Lemma~3.4]{TW}).
\begin{thm}[Tsukerman--Williams]
Let $W$ be of type $A$, and $u,v\in W$ with $u < v$. There exists a transposition $t$ such that $u \leq tv\lhd v$ and $u\lhd tu\leq v$. 
\end{thm}
 In \cite{CS}, Caselli and Sentinelli provide a characterization of the Coxeter groups satisfying the generalized lifting property (\cite[Theorem~6.9]{CS}).
  \begin{thm}[Caselli--Sentinelli]
  \label{generalizzatacaratt}
Let $(W, S)$ be a Coxeter system. The following are equivalent.
\begin{itemize}
\item For all $u,v\in W$ with $u < v$, there exists a reflection $t$ such that $u \leq tv\lhd v$ and $u\lhd tu\leq v$. 
\item The Coxeter system $(W, S)$ is finite and simply laced.
\end{itemize}  
 \end{thm} 
 The generalized lifting property plays a fundamental role in the study of Bruhat interval polytopes of the symmetric group in \cite{TW} (actually, this is the main reason why it has been discovered). In the study of Bruhat interval polytopes of arbitrary finite Coxeter group in the next sections, we make use of the following weaker property, which holds in every Coxeter group.

\begin{thm}[Weak generalized lifting property]\label{intersectcones}
Let $(W, S)$ be an arbitrary Coxeter system. Let $u,v\in W$ with $u < v$, $R=\{\alpha_t\in \Phi^+:\, u\leq tv\lhd v\}$, and $R'=\{\alpha_t\in \Phi^+:\, u \lhd tu \leq v \}$. Then $\Cone (R)\cap \Cone (R')\neq \{0\}$. 
\end{thm}
\begin{proof}
 If $\Cone (R)\cap \Cone (R')=\{0\}$, then by Corollary~\ref{C1C2} there exists a reflection ordering $\preceq $ such that all roots in $R$ are smaller than all roots in $R'$. In particular, there would be no increasing paths from $u$ to $v$, and this contradicts Proposition~\ref{unasola}. 
\end{proof}

The following problem asks for a result that is slightly stronger than the weak generalized lifting property. 

\begin{probl}
\label{cisono}
Given a Bruhat interval  $[u,v]$, we let  $R=\{\alpha_t\in \Phi^+:\, u\leq vt\lhd v\}$ and $R'=\{\alpha_t\in \Phi^+:\, u \lhd ut \leq v \}$. For which intervals, the intersections  $R\cap \Cone (R') $ and $R'\cap \Cone (R)$ are non-empty?
\end{probl}

We know that such intersections are non-empty for intervals in  finite simply laced Coxeter groups (by Theorem~\ref{generalizzatacaratt}). The next results show that this is the case also for dihedral Coxeter groups, and for all intervals of length 2. 
\begin{pro}
\label{valeperdiedrali}
Let $(W,S)$ be a dihedral Coxeter system. Let $u,v\in W$ with $u<v$,   $R=\{\alpha_t\in \Phi^+:\, u\leq vt\lhd v\}$, and $R'=\{\alpha_t\in \Phi^+:\, u \lhd ut \leq v \}$.  The intersections  $R\cap \Cone (R') $ and $R'\cap \Cone (R)$ are non-empty.
\end{pro}
\begin{proof}
Let $S=\{s,t\}$ and $m$ be the order of the product $st$ in $W$. 

We first suppose that $m$ is finite. For all $k=1,3,\ldots,2m-1$, let $s^{(k)}=sts\cdots s$ with $k$ factors, and  define $t^{(k)}$ similarly. Observe that $t^{(2m-k)}=s^{(k)}$, so the  reflections of $W$ are $$s^{(1)},s^{(3)},\ldots,s^{(2m-1)}.$$
 The corresponding roots appear in this order going from $\alpha_s$ to $\alpha_t$ (either clockwise or counterclockwise) in the 2-dimensional space $V$.

If $u$ is  the identity $e$, then the statement is trivial. Without loss of generality, we suppose  $u\neq e$ and that $s$ is the first letter of the unique reduced expression of $u$, so $u=sts \cdots $ with $k$ factors, for a certain $k\in  \{1, \ldots, m-1\}$. We may suppose that $v$ is not the longest element and that $s$ is the first letter of the unique reduced expression of $v$, since otherwise $t\in R\cap R'$. Let $v=sts\cdots $ with $j$ factors, for a certain $j\in  \{k+1, \ldots, m-1\}$. Since $R=\{t, s^{(2k+1)}\}$ and $R'=\{s, s^{(2j-1)}\}$, the assertion is proved.
 
Now suppose $m=\infty$. As above, let $s^{(k)}=sts\cdots s$ with $k$ factors, and   $t^{(k)}=tst\cdots t$ with $k$ factors, for $k$ odd. As $k$ varies over the odd integers, these reflections are different and cover all reflections of $W$. 
Let $\alpha,\beta,\gamma \in \Phi^+$: the root $\gamma$ belongs to $ \Cone (\{\alpha,\beta\})$ if and only if either
\begin{itemize}
\item $\alpha= s^{(k)}$, $\beta=s^{(j)}$, $\gamma=s^{(i)}$ with $k\leq i \leq j$, or
\item $\alpha= t^{(k)}$, $\beta=t^{(j)}$, $\gamma=t^{(i)}$ with $k\leq i \leq j$, or
\item $\alpha= s^{(k)}$, $\beta=t^{(j)}$, $\gamma=s^{(i)}$ with $k\leq i$, or
\item $\alpha= s^{(k)}$, $\beta=t^{(j)}$, $\gamma=t^{(i)}$ with $j\leq i$, 
\end{itemize}
(up to changing the role of $\alpha$ and $\beta$). Hence we can proceed   as in the case $m$ finite.
\end{proof}
 
 Statement (1) of the following proposition is well-known: we include it just for reference.
 
\begin{pro}\label{quadratisi}
Let $(W,S)$ be an arbitrary Coxeter system. Let $u,v,x,y$ and $t_i$, for $i=1,2,3,4$, be as in Figure \ref{fig1}, i.e. $u\lhd x\lhd v$, $u\lhd y\lhd v$, $x\neq y$, $x=t_1u$, $v=t_3x$, $v=t_4y$, $y=t_2u$. 
 We have 
 \begin{enumerate}
  \item 
  $Span\{\alpha_{t_1},\alpha_{t_3}\}=Span\{\alpha_{t_2},\alpha_{t_4}\}=Span\{\alpha_{t_1},\alpha_{t_2}\}=Span\{\alpha_{t_3},\alpha_{t_4}\}$,
  \item $\{\alpha_{t_1},\alpha_{t_2}\}\cap \Cone(\{\alpha_{t_{3}},\alpha_{t_4}\})\neq \emptyset$,
  \item $\{\alpha_{t_3},\alpha_{t_4}\}\cap \Cone(\{\alpha_{t_1},\alpha_{t_2}\})\neq \emptyset$.
  \end{enumerate}
\end{pro}
\begin{proof}For the proof of (1), see the proof of \cite[Lemma~3.1]{D1991}.
The statements (2) and (3) follows by Proposition~\ref{valeperdiedrali}, since Dyer's works on reflection subgroups (see \cite[Theorem~1.4]{D1991}, for example) imply that we may suppose $W$  dihedral.
\end{proof}

\section{Bruhat intervals are Coxeter matroids}
\label{intervallicoxetermatroids}
Let $(W,S)$ be an arbitrary finite Coxeter system. In this section, we prove that, for any $u,v\in W$ with $u\leq v$, the Bruhat  interval $[u,v]$ is a Coxeter matroid. This fact is a generalization of results of \cite{KW,TW}.


%

Recall Definition~\ref{politopo}, where $p$ is a  fixed appropriate base point. Following \cite{TW}, given a Bruhat interval $[u,v]$, we call the associated polytope $\Delta_{[u,v]}$  a \emph{Bruhat interval polytope}.

Next result is a direct geometric consequence of the weak generalized lifting property (Theorem~\ref{intersectcones}).
\begin{pro}
\label{catena}
Let $W$ be a finite Coxeter group, and $u,v\in W$ with $u\leq v$. Let $F$ be a face of the Bruhat interval polytope $\Delta_{[u,v]}(p)$. If $F$ contains $x(p)$ and $y(p)$ for some $x,y \in W$ with $u\leq x\leq y\leq v$, then there exists a complete chain $C$ from $x$ to $y$ such that $z(p) \in F$ for all $z\in C$.\end{pro}
\begin{proof}
 We proceed by induction on $\ell(y)-\ell(x)$. The cases $\ell(y)-\ell(x)=0,1$ are trivial, and so we suppose $\ell(y)-\ell(x)>1$.
 
 By the weak generalized lifting property, 
 there exist $r_1,\ldots,r_k, t_1,\ldots,t_h \in T$  with  $k,h>0$,  $x\lhd r_ix\leq y$ for all $i\in [k]$,  $x\leq t_i y\lhd y$ for all $i\in [h]$,  such that $\sum_i a_i \alpha_{r_i}=\sum_j b_j \alpha_{t_j}\neq 0$ with all $a_i,b_j> 0$.
 
Let $f\in V^*$ be a linear form such that $f=c$ is the equation of  the hyperplane containing $F$, and $f< c$ gives the halfspace containing $\Delta_{[u,v]}(p)\setminus F$, for a certain $c\in \mathbb R$. By Lemma~\ref{anticamera}, we have that there exist positive numbers $c_1,\ldots, c_k, d_1,\ldots,d_h$ such that $r_i(x(p))=x(p)+c_i\alpha_{r_i}$ for all $i=1,\ldots,k$ and $t_j(y(p))=y(p)-d_j\alpha_{t_j}$ for all $j=1,\ldots,h$. Since all points $r_i(x(p))$  and $t_j(y(p))$ belong to the polytope,    $f(\alpha_{r_i})\leq 0$ for all $i=1,\ldots,k$, and  $f(\alpha_{t_j})\geq 0$ for all $j=1,\ldots,h$. Thus $f(\sum a_i \alpha_{r_i})=f(\sum b_j \alpha_{t_j})=0$, and  $f(\alpha_{r_i})=f(\alpha_{t_j})=0$ for all $i=1,\ldots,k$ and $j=1,\ldots,h$; therefore, all points $r_ix(p)$ and $t_jy(p)$ lie in $F$. 
By the induction hypothesis, there is a complete chain $C'$ from $r_1x$ to $y$ such that $z(p)\in F$ for all $z\in C'$, and so  the complete chain $C=C'\cup \{x\}$ has the required property.
\end{proof}

In the study of the structure of the faces of the polytope $\Delta_{[u,v]}(p)$, we need some elementary properties of Bruhat intervals of length 2. 
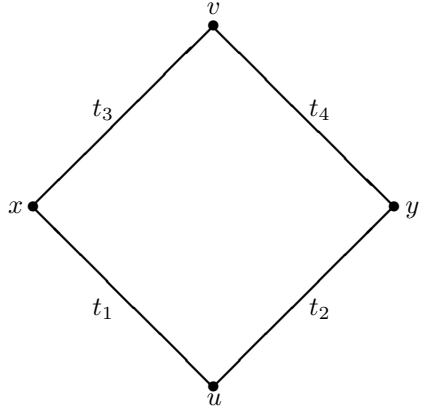
\begin{figure}
 \setlength{\unitlength}{8mm}
\begin{center}
\begin{picture}(0,6) 
\thicklines
\put(-0.1,-0.1){$\bullet$}
\put(-0.1,5.9){$\bullet$}
\put(-3.1,2.9){$\bullet$}
\put(2.9,2.9){$\bullet$}
\put(0,0){\line(1,1){3}}
\put(0,0){\line(-1,1){3}}
\put(-3,3){\line(1,1){3}}
\put(0,6){\line(1,-1){3}}
\put(-0.1,-0.3){$u$}
\put(-0.1,6.2){$v$}
\put(3.2,2.9){$y$}
\put(-3.4,2.9){$x$}
\put(-2,1.2){$t_1$}
\put(1.6,1.2){$t_2$}
\put(-2,4.5){$t_3$}
\put(1.6,4.5){$t_4$}
\end{picture}
\end{center}
\caption{Bruhat interval of length 2 \label{fig1}}
\end{figure}
\begin{lem}\label{lemma1}
 Let $u,v,x,y$ and $t_i$, $i=1,2,3,4$ be as in Figure~\ref{fig1}, i.e. $u\lhd x\lhd v$, $u\lhd y\lhd v$, $x\neq y$, $x=t_1u$, $v=t_3x$, $v=t_4y$, $y=t_2u$.  Let $f\in V^*$ and $p\in V$ be such that $(p,\alpha_s)<0$ for all $s\in S$.   
 \begin{enumerate}
  \item 
  If $f$ is equal to $c$ on any three of the four points $u(p),x(p),y(p),v(p)$, then it is equal to $c$ on the remaining one also.
  \item If $f(u(p))=f(x(p))=c$ and $f(v(p))< c$, then $f(y(p))< c$.
  \item If $f(u(p))=f(x(p))=c$ and $f(y(p))< c$, then $f(v(p))< c$.
 \end{enumerate}
\end{lem}
\begin{proof}
For $i=1,2,3,4$, we let $\alpha_i=\alpha_{t_i}$. By Lemma \ref{anticamera}, there exist $k_1,k_2,k_3,k_4>0$ such that 
  
  \begin{enumerate}
  \item [(i)]$x(p)=u(p)+k_1\alpha_1$,
  \item [(ii)]$y(p)=u(p)+k_2\alpha_2$, 
  \item [(iii)]$v(p)=x(p)+k_3\alpha_3$,
  \item [(iv)]$v(p)=y(p)+k_4\alpha_4$.
  \end{enumerate}

Since a functional is constant on two points if and only if it is zero on their difference, (1) is straightforward by (1) of Proposition \ref{quadratisi}. 

If $f(u(p))=f(x(p))=c$ and $f(v(p))< c$, then $f(\alpha_1)=0$ by (i), $f(\alpha_3)< 0$ by (iii), and $f(y(p))\neq c$ by (1). Toward a  contradiction, suppose  $f(y(p))>c$. By (ii) and (iv),  we obtain  $f(\alpha_2)>0$ and $f(\alpha_4)<0$. By the weak generalized lifting property (Theorem~\ref{intersectcones}), there exist nonnegative $a_1,a_2,a_3,a_4$ such that
 \[
  a_1\alpha_1+a_2\alpha_2=a_3\alpha_3+a_4\alpha_4\neq 0.
 \]
We  get a contradiction by applying $f$ to this equality, and so (2) holds.
 
 If $f(u(p))=f(x(p))=c$ and $f(y(p))< c$, then $f(\alpha_1)=0$ by (i), $f(\alpha_2)< 0$ by (ii), and $f(v(p))\neq c$ by (1).  Toward a contradiction, suppose  $f(v(p))>c$. By (iii) and (v), we obtain $f(\alpha_3)>0$ and $f(\alpha_4)>0$. As above, we get a contradiction by the weak generalized lifting property, and hence (3) holds.
 \end{proof}


\begin{cor}
\label{tutto}
Let $W$ be a finite Coxeter group, and $u,v\in W$ with $u\leq v$. Let $F$ be a face of the Bruhat interval polytope $\Delta_{[u,v]}(p)$. If $F$ contains $x(p)$ and $y(p)$ for some $x,y\in W$ with $u\leq x\leq y\leq v$, then $F$ contains all points $z(p)$, as $z$ varies in the interval $[x,y]$.
\end{cor}
\begin{proof}
By  Proposition~\ref{catena}, there exists a complete chain $C$ from $x$ to $y$ such that $z(p) \in F$, for all $z\in C$. Let $z\in [x,y]$, and let $C'$ be a complete chain $C$ from $x$ to $y$ containing $z$. By Proposition~\ref{collegate},  there exist $k\in \mathbb N$ and a sequence $C=C_1, C_2, \ldots, C_k=C'$ of complete chains from $x$ to $y$ such that $C_i$ and $C_{i+1}$ differ by one element, for all $i=1,2, \ldots, k-1$. Therefore, we may conclude by (1) of Lemma~\ref{lemma1}.
\end{proof}
We are now ready to prove the main result of this section.
\begin{thm}
\label{sonoCM}
Let $(W,S)$ be a finite Coxeter system. For all $u,v\in W$ with $u\leq v$, the interval $[u,v]$  is a Coxeter matroid.
\end{thm}
\begin{proof}
Recall the Gelfand--Serganova Theorem (Theorem~\ref{defCM}). We show that every edge in the Bruhat interval polytope $\Delta_{[u,v]}(p)$ is parallel to a root. Let $x(p)$ and $y(p)$ be the two vertices of an edge $E$ of $\Delta_{[u,v]}(p)$. We recall that $E$ does not contain $z(p)$, for $z\notin\{x,y\}$, by Remark~\ref{vero}. We show that $x$ and $y$ are in covering relation. Suppose to the contrary that $x$ and $y$ are not  in covering relation. 

Let $H$ be a hyperplane with equation $f=a$ such that the edge $E$ belongs to $H$ and $\Delta_{[u,v]}(p)\setminus E$ is contained in the halfspace $f< a$. Set $U=\{t\in T : x\lhd tx\leq v \text{ or } y\lhd ty\leq v\}$ and $D=\{r\in T : u\leq rx\lhd x \text{ or } u\leq ry\lhd y\}$. Observe that the condition that $x$ and $y$ are not in covering relation implies that $tx\neq y$ for all $t\in D\cup U$. In particular, if $t\in U$ and $x\lhd tx$, then $f(tx(p))<a$. By Lemma~\ref{anticamera}, $tx(p)=x(p)+c\alpha_t$ with $c>0$, and hence $f(\alpha_t)<0$. The same argument shows 
$f(\alpha_t)<0$ for all $t\in U$, and similarly  $f(\alpha_r)>0$ for all $r\in D$.
By Proposition~\ref{reflecorder}, there exists a reflection ordering $\preceq$ such that $r\preceq t$ for all $r\in D$ and all $t\in U$.  Hence $x$ and $y$ are not comparable since otherwise there should be an increasing complete chain connecting $x$ and $y$ (by Proposition~\ref{unasola}) and this contradicts $r\preceq t$ for all $r\in D$ and all $t\in U$. 
Let $C_1$ be the unique  increasing complete chain from $u$ to $x$ and $C_2$ be the unique  increasing complete chain from $x$ to $v$. The condition $r\preceq t$ for all $r\in D$ and $t\in U$ implies that the concatenation of $C_1$ and $C_2$ is an increasing chain from $u$ to $v$ passing through $x$. Similarly, there is an increasing chain from $u$ to $v$ passing through $y$, and this contradicts the unicity guaranteed by Proposition~\ref{unasola} (since $x$ and $y$ are not comparable). 

Therefore, $x$ and $y$ are in covering relation, $y=tx$ for some reflection $t$, and the edge $E$ is parallel to the root $\alpha_t$. 
\end{proof}

\section{Structure of Bruhat interval polytopes}
\label{struttura}
In this section, we provide some results  on the structure of Bruhat interval polytopes and, in particular, we show that faces of a Bruhat interval polytope are themselves Bruhat interval polytopes. As usual, we fix a point $p\in V$ such that  $(p,\alpha_{s})<0$ for all $s\in S$.  
\begin{thm}
\label{faccesonointervalli}
Let $(W,S)$ be a finite Coxeter system and $u,v\in W$ with $u\leq v$.  Then every face of the Bruhat interval polytope $\Delta_{[u,v]}(p)$ is a Bruhat interval polytope $\Delta_{[x,y]}(p)$, for some $u\leq x\leq y\leq v$.
\end{thm}
\begin{proof}Let $F$ be a face of $\Delta_{[u,v]}(p)$ and $A\subset [u,v]$ be such that $F$ is the convex hull of $\{z(p):\, z\in A\}$. Since edges of $F$ are edges of $\Delta_{[u,v]}(p)$, the  Gelfand--Serganova Theorem (Theorem~\ref{defCM}) and Theorem~\ref{sonoCM} imply that also $A$ is  a Coxeter matroid; in particular, $A$ has a maximum $y$ and a minimum $x$ with respect to the Bruhat order $\leq$, which coincides with $\leq^e$. The result follows by Corollary~\ref{tutto}.
\end{proof}
In the case $W$ is the symmetric group, Theorem~\ref{faccesonointervalli} specialises to the main result of \cite[\S4.1]{TW}.

When  $u<v$, we let $\Phi(u,v)$ be the set of positive roots corresponding to the covering relations between elements in $[u,v]$, and  $U_{u,v}=\Span(\Phi(u,v))$.
\begin{pro}
\label{uguale}
Let $(W,S)$ be a finite Coxeter system. Let $u,v\in W$ with $u\leq v$, and $C$ be any maximal chain from $u$ to $v$. Let  $\Phi_{\at}(u,v)$, $\Phi_{\coat}(u,v)$, $\Phi_{C}$ be the sets of positive roots corresponding to  the covering relations between $u$ and the atoms of $[u,v]$, between $v$ and the coatoms of $[u,v]$, and between elements in $C$, respectively. Then 

\[
 U_{u,v}=\Span(\Phi_{\at}(u,v))=\Span(\Phi_{\coat}(u,v))=\Span(\Phi_C).
\]
\end{pro}
\begin{proof}We first show that  $U_{u,v}=\Span(\Phi_{\at}(u,v))$. Let $u<x\lhd w\leq v$. Since $x\neq u$, the covering relation $x\lhd w$ is an ``upper'' edge of a Bruhat interval of rank 2 contained in $[u,v]$. By (1) of Proposition \ref{quadratisi},  the root $\alpha_{x\lhd w}$ is a linear combination of the roots corresponding to the ``lower'' edges of such interval. By iterating this procedure, we show that $\alpha_{x\lhd w}$ is a linear combination of roots in $\Phi_{\at}(u,v)$, and so $\Span(\Phi_{\at}(u,v))=U_{u,v}$. The proof of $\Span(\Phi_{\coat}(u,v))=U_{u,v}$ is analogous.  
 Finally, $\Span(\Phi_C)=U_{u,v}$ follows similarly from
  Proposition~\ref{collegate} and (1) of Proposition \ref{quadratisi}.
\end{proof}

\begin{rem}
Observe that Proposition~\ref{uguale} implies that $\dim(U_{u,v})\leq \min \{\, \# \text{atoms}, \,\# \text{coatoms}, \,\ell(v)-\ell(u)\}$. 
Let $W_{u,v}$ denote the reflection subgroup generated by  the reflections labeling covering relations between elements in $[u,v]$. The group $W_{u,v}$ is itself a Coxeter group and its rank is  $\dim(U_{u,v})$. By Lemma \ref{anticamera}, a linear form is constant on the Bruhat interval polytope $\Delta_{[u,v]}$ if and only if  it vanishes on $U_{u,v}$, which implies $\dim(\Delta_{[u,v]})=\dim(U_{u,v})=\rank(W_{u,v})$.
\end{rem}

%
%
%
%
%

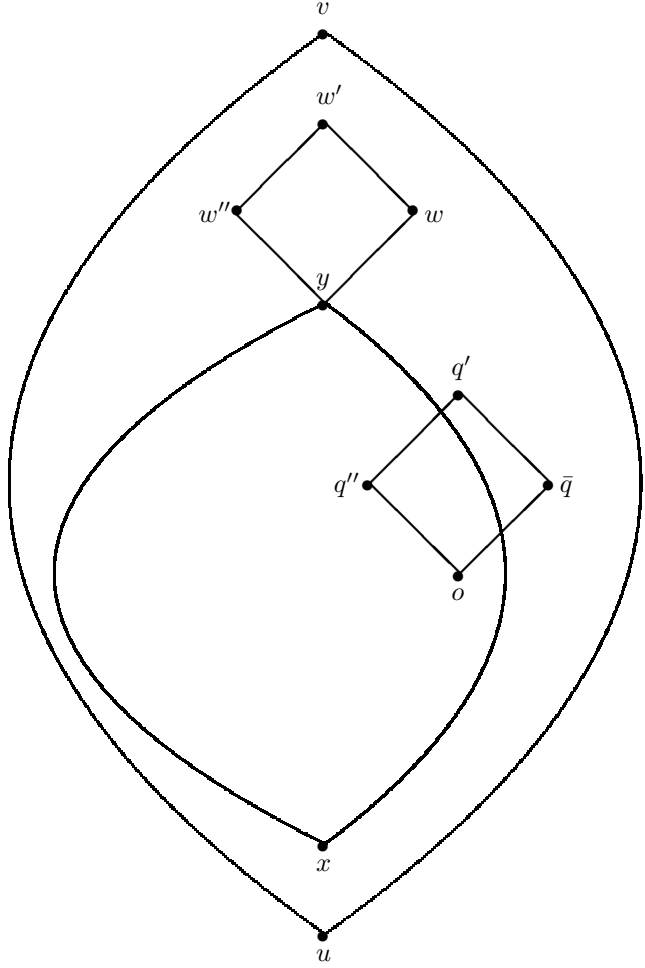
\begin{figure}
 \setlength{\unitlength}{12mm}
\begin{center}
\begin{picture}(0,10) 
\thicklines
\qbezier{(0,0)(-7,5)(0,10)}
\qbezier{(0,0)(7,5)(0,10)}
\qbezier{(0,1)(-6,4)(0,7)}
\qbezier{(0,1)(4,4)(0,7)}
\put(-0.1,-0.1){$\bullet$}
\put(-0.1,9.9){$\bullet$}
\put(-0.1,0.9){$\bullet$}
\put(-0.1,6.9){$\bullet$}

\put(-0.1,8.9){$\bullet$}
\put(0.9,7.95){$\bullet$}
\put(-1.05,7.95){$\bullet$}
\put(0.4,4.9){$\bullet$}
\put(2.4,4.9){$\bullet$}
\put(1.4,5.9){$\bullet$}
\put(1.4,3.9){$\bullet$}

\put(0,7){\line(1,1){1}}
\put(-1,8){\line(1,-1){1}}
\put(-1,8){\line(1,1){1}}
\put(0,9){\line(1,-1){1}}

\put(1.5,4){\line(1,1){1}}
\put(0.5,5){\line(1,-1){1}}
\put(0.5,5){\line(1,1){1}}
\put(1.5,6){\line(1,-1){1}}

\put(-0.1,-0.3){$u$}
\put(-0.1,10.2){$v$}
\put(-0.1,7.2){$y$}
\put(-0.1,0.7){$x$}
\put(-0.1,9.2){$w'$}
\put(-1.4,7.9){$w''$}
\put(1.1,7.9){$w$}

\put(1.4,3.7){$o$}
\put(1.4,6.2){$q'$}
\put(0.1,4.9){$q''$}
\put(2.6,4.9){$\bar q$}
\end{picture}
\end{center}
\caption{Proof of Lemma \ref{4.18giusto} \label{fig2}}
\end{figure}

\begin{lem}
\label{4.18giusto}
Let $(W,S)$ be a finite Coxeter  system, and  $x,y,u,v\in W$ be such that $u\leq x\leq y\leq v$.  Suppose that $f \in V^*$ and $c\in \mathbb R$ are such that $f(z(p))=c$ for all $z\in [x,y]$. If $f(w(p))\leq c$ for all $w\in [u,v]$ covering  $y$, then  $f(q(p))\leq c $ also for all $q\in [u,v]$ covering any element  $o\in[x,y]$. 
\end{lem}
\begin{proof}
The reader may refer to Figure \ref{fig2}. 
First we show the result in the case $f(w(p))<c$ for all $w$ such that $y\lhd w\leq v$.
 
We use induction on $\ell(y)-\ell(o)$. If $\ell(y)-\ell(o)=0$, then $o=y$ and the assertion is true by the hypothesis.

Suppose $\ell(y)-\ell(o)>0$. 
We choose any basis $\{f_1,\ldots,f_n\}$ of $V^*$ with $f_1=-f$ and consider the reflection ordering $\preceq $ associated with this basis (see  Proposition~\ref{reflecorder}). Given a root $\alpha$ labeling a covering relation in $[x,y]$ and a root $\beta$ labeling a covering relation of the form $y \lhd w$ with $w\in[u,v]$, we have 
\begin{eqnarray}
\label{tuttepiupiccole}
\alpha \prec \beta 
\end{eqnarray}
by the two following facts:
\begin{itemize}
\item $f(\alpha)=0$, since $f$ is constant on $[x,y]$,
\item $f(\beta)<0$, since $f(y(p)) > f(w(p))$ and $w(p)= y(p) +k\beta$ for some $k>0$ by Lemma~\ref{anticamera}.  
\end{itemize}
Now we choose $\bar q \in [u,v]\setminus [x,y]$ covering $o$ in such a way  that  $\alpha_{o\lhd \bar q}$ is as small as possible in the reflection ordering $\preceq$ (of course, when there is no such $\bar q$, there is nothing to prove).  We get the assertion if we show that $f(\alpha_{o\lhd \bar q}) \leq 0$: indeed, if $f(\alpha_{o\lhd \bar q}) \leq 0$, then $f(\alpha_{o\lhd q})\leq f(\alpha_{o\lhd \bar q})\leq0$ for all $q\in [u,v]\setminus[x,y]$ such that  $o\lhd q$ (since $\alpha_{o\lhd q}\succeq \alpha_{o\lhd \bar q}$), and so $f(q(p))\leq c$.  


Let us prove  $f(\alpha_{o\lhd \bar q})\leq 0$. If $\alpha_{o\lhd \bar q} \in U_{x,y}$, then $f(\alpha_{o\lhd \bar q})=0$, and so  we may suppose $\alpha_{o\lhd \bar q} \notin U_{x,y}$.

Consider the increasing complete chain from $o$ to $y$ and the increasing complete chain from $y$ to $v$ (which exist and are unique by Proposition~\ref{unasola}). By Equation \eqref{tuttepiupiccole}, the concatenation of these two chains is an increasing complete chain from $o$ to $v$, and we recall that there are no other  increasing complete chains from $o$ to $v$ by Proposition~\ref{unasola}. Now consider the increasing complete chain from $\bar q$ to $v$. Denote by $q'$ the element covering $\bar q$ in this chain, and by $q''$ the element such that $[o,q']=\{o,\bar q,q'',q'\}$. Necessarily, $\alpha_{\bar q\lhd q'}\prec \alpha_{o\lhd \bar q}$, since otherwise there would be another  increasing complete chain from $o$ to $v$; thus the increasing complete chain from $o$ to $q'$ must be $o\lhd q''\lhd q'$. This implies that $q''\in [x,y]$, since otherwise $\alpha_{o\lhd \bar q} \prec \alpha_{o\lhd q''}$ by the minimality of $\alpha_{o\lhd \bar q}$ and we would have $\alpha_{\bar q\lhd q'} \prec \alpha_{o\lhd \bar q} \prec \alpha_{o\lhd q''} \prec \alpha_{q''\lhd q'}$, which implies both $\alpha_{\bar q \lhd q'} \notin \Cone(\alpha_{o\lhd \bar q},  \alpha_{o\lhd q''})$ and $\alpha_{q''\lhd q'} \notin \Cone(\alpha_{o\lhd q},  \alpha_{o\lhd q''})$, a  contradiction by  Proposition \ref{quadratisi}.

Thus $q''\in [x,y]$,  $\alpha_{o\lhd q''}$ belongs to $U_{x,y}$, and hence $f(\alpha_{o\lhd q''})=0$. This, together with $f(\alpha_{q''\lhd q'})\leq 0$, which holds by the induction hypothesis, implies $f(\alpha_{o\lhd \bar q})\leq 0$ by  Lemma~\ref{lemma1}.

Therefore,  we may suppose the additional condition that there exists $w\in [u,v]$ covering  $y$ such that $f(w(p))=c$ and we proceed by induction on $\ell(v)-\ell(y)$, the result being trivial for $\ell(v)-\ell(y)=0$. 
Proposition~\ref{uguale} implies $U_{x,w}= U_{x,y}+ \mathbb R \alpha_{y\lhd w}$, and hence $f(z(p))=c$, for all $z\in [x,w]$. The result follows if $\ell(v)-\ell(y)=1$, i.e. $w=v$. Otherwise, if $\ell(v)-\ell(y)>1$, let $w'\in [u,v]$ such that $w\lhd w'$ and consider the interval $[y,w']$. Let $w''$ be such that $[y,w']=\{y,w,w',w''\}$.
Since $f (w''(p))\leq c$ by hypothesis,  $f(\alpha_{y\lhd w''}) \leq 0$ by Lemma~\ref{anticamera}; by Lemma~\ref{lemma1}, the conditions $f(\alpha_{y\lhd w''}) \leq 0$ and $f(\alpha_{y\lhd w}) = 0$ imply that $f(\alpha_{w\lhd w'}) \leq 0$. 

Hence, $f(w'(p))\leq c$, for all  $w' \in [u,v]$ covering $w$. Since $\ell(v)-\ell(w)=\ell(v)-\ell(y)-1$, by the induction hypothesis $f(\tilde q(p))\leq c$ for all $\tilde q \in [u,v]$ covering an element in $[x,w]$. The  result follows since, when $q$ covers an element in $[x,y]$, it  covers also an element in $[x,w]$, since $[x,y]\subset [x,w]$.
\end{proof}

\begin{rem}
\label{4.18}
Lemma~\ref{4.18giusto} implies \cite[Lemma~4.18]{TW}, and holds for all finite Coxeter groups. The proof of Lemma \ref{4.18giusto} is more involved than the proof of \cite[Lemma 4.18]{TW} not only because it holds for all finite Coxeter groups and not just for the symmetric group, but also because there is a gap in the proof of \cite[Lemma~4.18]{TW}. Precisely, given $q,q',z\in W$ such that $q\lhd q'$ and $q\lhd z$, there could be no element $z'\in W$ such that $z\lhd z'$ and $q'\lhd z'$. This is the case also in the symmetric group (for example, take $q=2143=s_1s_3$, $q'=2341=s_1s_2s_3$, and $z=4123=s_3s_2s_1$). However, the statement of \cite[Lemma~4.18]{TW} is true thanks to Lemma~\ref{4.18giusto}, and so all the results in \cite{TW} that are consequences of \cite[Lemma~4.18]{TW} are true.
\end{rem}

Our final result is a description of all faces of a Bruhat interval polytope. Both the result and the proof are natural extensions of those of \cite[Theorem~4.19]{TW}.
\begin{pro}
Let $(W,S)$ be a finite Coxeter  system, and  $x,y,u,v\in W$ be such that $u\leq x\leq y\leq v$.  Suppose that $f \in V^*$ and $c\in \mathbb R$ are such that $f (z(p))=c$ for all $z\in [x,y]$. If $f (w(p))<c$ for all $w\in [u,v]$ such that either $y\lhd w$ or $w\lhd x$, then  $f(q(p))<c$ for all $q\in [u,v]\setminus [x,y]$. 
\end{pro}
\begin{proof}
Lemma~\ref{4.18giusto} implies,  in particular, that $f(q(p))\leq c$  for all $q\in [u,v]$ covering $x$. Thus, by the hypothesis, $f(w(p))\leq c $  for all $w$ in the set $A= \{w\in [u,v]:  \text{either $w\lhd  x$ or $x\lhd w$ 
} \}$.
By Theorem~\ref{faccesonointervalli}, in  the Bruhat interval polytope $\Delta_{[u,v]}$, every edge  containing $x(p)$ joins $x(p)$ to a point $w(p)$ for some $w\in A$.  Since the Bruhat interval polytope $\Delta_{[u,v]}(p)$ is contained in the cone centered in $x(p)$ and generated by the edges emanating from $x(p)$, we have that $\Delta_{[u,v]}(p)$ is contained in the halfspace of equation $f\leq c$. Thus the hyperplane of equation $f = c$ cuts a face $F$ of  $\Delta_{[u,v]}(p)$ containing $\Delta_{[x,y]}(p)$. By Theorem~\ref{faccesonointervalli}, the face $F$ is a Bruhat interval polytope; since it  contains  neither elements $w(p)$ with $w$ covering $y$ nor elements $w(p)$ with $w$ covered by $x$, $F$ must be $\Delta_{[x,y]}(p)$, and the result follows.
\end{proof}
We state explicitly (without proof) the following straightforward consequence of the previous results.
\begin{cor}
Let $(W,S)$ be a finite Coxeter  system, and  $x,y,u,v\in W$ be such that $u\leq x\leq y\leq v$. The Bruhat interval polytope $\Delta_{[x,y]}(p)$ is a face of the Bruhat interval polytope $\Delta_{[u,v]}(p)$ if and only if  there exist $f\in V^*$ and $c\in \mathbb R$ such that 
\begin{itemize}
\item  $f(z(p))=c$ for all $z\in [x,y]$,
\item   $f(w(p))<c$ for all $w\in [u,v]$ such that either $y\lhd w$ or $w\lhd x$.
\end{itemize} 
\end{cor}

\section{Parabolic Bruhat intervals are Coxeter matroids}
Let $(W,S)$ be an arbitrary finite Coxeter system, and $J\subseteq S$. In this section, we generalize the results in Sections~\ref{intervallicoxetermatroids}-\ref{struttura} to the parabolic setting.

\medskip

We first need a few geometric results, which might be of some interest in their own.  Fix $n, r \in \mathbb{N}$ with $n,r \geq 1$, and $\epsilon\in \mathbb{R}$ with  $\epsilon>0$. Let $x_1(t),x_2(t),\dots,x_r(t)\in \mathbb{R}^n$ be points that vary continuously for $t\in [0,\epsilon)$, i.e.  we have a continuous map $f:[0,\epsilon)\to (\mathbb{R}^{n})^r$, $t\mapsto f(t)=(x_1(t),x_2(t),\dots,x_r(t))$.

For every $I\subseteq \{1, \ldots , r\}$ and $t\in [0,\epsilon)$, we denote by $F_I(t)$  the convex hull in $\mathbb{R}^{n}$ of the points $\{x_i(t) : i\in I\}$. For short, we let $P(t)=F_{\{1, \ldots, r\}}(t)$.

We assume that, for every $t\in (0,\epsilon)$, the points $x_i(t)$ are all distinct vertices of $P(t)$. We also assume that, if $F_I(t)$ is a face of $P(t)$ of dimension $k$ for some $t\in(0,\epsilon)$, then $F_I(t)$ is a face of $P(t)$ of dimension $k$ for all $t\in(0,\epsilon)$. Notice that its ``limit'' $F_I(0)$ has dimension less than or equal to $k$ and need not  be a face.

\begin{lem} \label{lem:limboundary}
Let $F(t)$ be a face of $P(t)$, for all $t\in (0,\epsilon)$. 
The proper faces of $F(0)$ are contained in the union of the limits of the proper faces of $F(t)$ (in other words, the boundary of the limit is contained in the limit of the faces in the boundary).
\end{lem}
\begin{proof} 

	
Let $p$ be a point of the boundary of $F(0)$. If $I$ is such that $F(t)=F_I(t)$, then there exist (possibly non unique) $a_i\in [0,1]$, with $\sum_{i\in I} a_i=1$,  such that $p= \sum_{i\in I}a_i x_i(0)$. 
Let $p(t)=\sum_{i\in I}a_i x_i(t)$ and $b(t)=\frac{1}{|I|}\sum_{i\in I} x_i(t)$ (i.e. the barycenter of $F(t)$).

If there was a sequence $t_n \rightarrow 0$ such that $p(t_n)=b(t_n)$, then $b(0)=\frac{1}{|I|}\sum_{i\in I} x_i(0) =\lim b(t_n)=\lim p(t_n)=p$, which is impossible since $b(0)$ is in the interior of $F(0)$. Hence, there exists $\mu \in (0,\epsilon]$ such that for each $t\in[0,\mu)$ we can consider the half-line $r_t$ from $b(t)$ through $p(t)$. The half-line $r_t$  intersects the boundary of $F(t)$ in  one point only, say $q(t)$. As $t$ tends to $0$, the point $q(t)$ tends to $p$ since $p$ belongs to both  $r_0$ and the boundary of $F(0)$.  
%

There exists a proper face $F_I(t)$ of $F(t)$ such that one can find a sequence $t_n \rightarrow 0$ such that the sequence of points $q(t_n)$ tends to $p$ and satisfies $q(t_n)\in F_I(t_n)$ for all $n$. We need to show  $p\in F_I(0)$. If not, the distance $d(p,F_I(0))$ would be strictly greater than $0$. On the other hand
$d(p,F_I(0)) \leq d(p,q(t_n)) + d(q(t_n),F_I(0))$. The first summand tends to $0$. Also the second one: indeed, if $q(t_n)=\sum_{i\in I}a_i(n)x_i(t_n)$ with $a_i(n)\in [0,1]$, then $d(q(t_n),F_I(0))\leq d(q(t_n), \sum_{i\in I}a_i(n)x_i(0) ) \leq |I|\max_{i\in I} d(x_i(t_n), x_i(0) )$.
\end{proof}

\begin{pro}
\label{paralleli}
Suppose that the directions of the edges of $P(t)$ are constant for $t\in (0,\epsilon)$, i.e., if $F_I(t)$, with $|I|=2$,  is an edge of $P(t)$, then $F_I(t')$ and $F_I(t'')$ are parallel for all $t',t''\in (0,\epsilon)$. 
For every $k$-dimensional face $F$ of $P(0)$, there exists a $k$-dimensional face $F(t)$ of $P(t)$ for $t\in (0,\epsilon)$ such that $F(0)=F$. 
\end{pro}
\begin{proof}
We may suppose that the affine space of $P(t)$ is constant (otherwise, we can change $x_i(t)$ with $x_i(t)-x_1(t)$, for all $i$) and hence full dimensional. 

Observe that the hypothesis on the edges implies that, if $F_I(t)$ is a face of $P(t)$ and $H_I(t)$ is the affine span  of $F_I(t)$ for $t\in (0,\epsilon)$, then $H_I(t')$ and $H_I(t'')$ are parallel for all $t',t''\in (0,\epsilon)$. 

We may suppose $\dim F \geq 1$, since  the statement is obviously true for $\dim F=0$. 
	
For $t\in (0,\epsilon)$, consider a face $F(t)$ of $P(t)$ of minimal dimension such that $F(0)$ meets 
the interior of $F$. Such a face  exists by  Lemma~\ref{lem:limboundary}; 
moreover, not only $F(0)$ but also the interior of $F(0)$ meets the interior of $F$, by minimality. 

Let $p$ be a point that belongs to both the interior of $F(0)$ and the interior of $F$. We make the following observations.
\begin{itemize}
 \item $F \supseteq F(0)$. Indeed, if $H$ is a hyperplane that cuts $F$ from $P(0)$, then $H\supseteq F(0)$, otherwise $p$ would not be in the interior of $F(0)$.
 \item $\dim F(0)=\dim F(t)$. Indeed, if not,  the dimension of $F(0)$ would be strictly smaller than the dimension $h$ of $F(t)$; hence,  there would exist a sequence $t_n \rightarrow 0$  such that $F(t_n)$ does not contain an $h$-dimensional ball of radius $1/n$. Let $p=\sum_ia_ix_i(0)$ and, for all $n$, choose a point  $p_n$ in the boundary of $F(t_n)$ at distance smaller than $1/n$ from $\sum_ia_ix_i(t_n)$:  the sequence of points $p_n$ in the boundary of $F(t_n)$ converges to $p$, contradicting the minimality of $F(t)$.
 \item $\dim F(t) \geq \dim F$. Indeed, if not, at least one of the  hyperplanes of the facets of $P(t)$ that determine $F(t)$ would be independent from  the hyperplanes $H_1, H_2,\dots, H_s$ of the facets of $P(0)$ that determine $F$. Denote such hyperplane by  $H(t)$, for $t\in (0,\epsilon)$. By the observation at the beginning of this proof,   in the limit at $t=0$ we get a hyperplane that supports
 $P(0)$,  contains $p$, but  is independent from the hyperplanes $H_1, H_2,\dots, H_s$. This is a contradiction. 
\end{itemize}

Therefore, $\dim F \geq \dim F(0)=\dim F(t) \geq \dim F$, so  $\dim F(t)=\dim F$ for all $t\in [0,\epsilon)$.

Let us prove  $F =F(0)$ by contradiction. Since we already know $F\supseteq F(0)$, we suppose that  the containment is strict. Choose a point in the interior of $F(0)$ (and so also in the interior of $F$) and a point of $F$ outside $F(0)$. Since $\dim F(0)=F$, the line segment joining these two points intersects  the boundary of $F(0)$ in a point $q$ that is internal to $F$. By Lemma~\ref{lem:limboundary}, the point $q$ belongs to the limit of a proper face of $F(t)$, and this contradicts   the minimality of $F(t)$.
%
%
\end{proof}

The restriction imposed on the directions of the edges in Lemma~\ref{paralleli} is not superflous. Indeed, let $n=2$, and, for $t\in [0,1)$, let  $P(t)$ be the convex hull of the points $x_1=(-1,0)$, $x_2=(0,-1)$, $x_3=(1,0)$, $x_4(t)=(0,t)$. The convex hull of $x_1$ and $x_3$ is a face $F$ of $P(0)$; nevertheless, for $t\in (0,1)$, there is no face $F(t)$ of $P(t)$  such that $F(0)=F$.

\medskip
We can now tackle the case of an arbitrary parabolic interval.
\begin{thm}
\label{sonoCMparabolici}
Let $(W,S)$ be a finite Coxeter system, $J\subseteq S$, and  $A,B\in W/W_J$ with $A\leq B$. Then  the parabolic interval  $[A,B]$  is a Coxeter matroid. Furthermore,
every face of the polytope $\Delta_{[A,B]}(p)$ is a  polytope $\Delta_{[C,D]}(p)$, for some $A\leq C\leq D\leq B$.
\end{thm}
\begin{proof}
The first statement follows by Theorem \ref{sonoCM} and \cite[Proposition 4.1]{SVZ}. It is also a direct consequence of the second statement, which implies in particular that every edge of $\Delta_{[A,B]}(p)$ is parallel to the root associated to the covering relation $C\lhd D$.

Let us prove the second statement. 

Let $v_{\max}$ and  $u_{\min}$ be the maximum of $B$ and  the minimum of $A$, respectively.
Consider the full interval $[u_{\min},v_{\max}]\subseteq W$, which is the disjoint union of the left cosets in $[A,B]$.

Fix two points $p,p'\in V$ with  
$(p,\alpha_{s})
  \begin{cases}
 <0 &  \text{for $s\notin J$}\\
 =0 &  \text{for $s\in J$}
\end{cases}
$ and  $p'$  in the anti-fundamental chamber, i.e. $(p',\alpha_{s})
<0 $ for all $s\in S$.  
For $t\in[0,1)$, let $p_t= (1-t)p + tp'$ and $P(t)$ be the convex hull of $\{z(p_t) : z\in [u_{\min} , v_{\max}] \}$. 
Notice  that
\begin{itemize}
\item $P(t)=\Delta_{[u_{\min},v_{\max}]}(p_t)$ for $t\in(0,1)$,
\item $P(0)=\Delta_{[A,B]}(p)$,
\item for every $t\in (0,1)$, the points $z(p_t)$ are all distinct vertices of $P(t)$, by Remark~\ref{vero},
\item  if $F_I(t)$ is a face of $P(t)$ of dimension $k$ for some $t\in(0,1)$, then $F_I(t)$ is a face of $P(t)$ of dimension $k$ for all $t\in(0,1)$, where $F_I(t)$ is the convex hull of $\{z(p_t): z\in I\}$ for $I\subseteq [u_{\min},v_{\max}]$, by Remark~\ref{vero2},
\item  the directions of the edges of $P(t)$ are constant for $t\in (0,1)$, by Theorem~\ref{faccesonointervalli}.
\end{itemize}

Let $F$ be a $k$-dimensional face of  $\Delta_{[A,B]}(p)$. By Proposition~\ref{paralleli}, there exists a $k$-dimensional face $F(t)$ of $\Delta_{[u_{\min},v_{\max}]}(p_t)$  for $t\in (0,1)$ whose limit $F(0)$ is the face $F$. By Theorem~\ref{faccesonointervalli}, there exists an interval $[x,y]  \subseteq [u_{\min},v_{\max}]$   such that $F_{[x,y]}(t)=F(t)$, and so $F(0)$ is the convex hull of $\{z(p): z\in [x,y]\}$. The assertion follows since $\{z(p): z\in [x,y]\}=\{z(p): zW_J\in [xW_J,yW_J]\}$.
%
\end{proof}



\end{document}